\documentclass[accepted]{ci4ts2024}
\pdfoutput=1

\usepackage{presets_philip}

\author[1,2]{Philip Boeken}
\author[1]{Joris M.\ Mooij}
\affil[1]{%
    Korteweg-de Vries Institute for Mathematics\\
    University of Amsterdam
}
\affil[2]{%
    Booking.com, Amsterdam, The Netherlands
}

\newcommand{\notefn}[1]{\footnote{\red{#1}}}
\title{Dynamic Structural Causal Models}

\begin{document}
\maketitle

\begin{abstract}
	We study a specific type of SCM, called a \emph{Dynamic Structural Causal Model} (DSCM), whose endogenous variables represent functions of time, which is possibly cyclic and allows for latent confounding. As a motivating use-case, we show that certain systems of Stochastic Differential Equations (SDEs) can be appropriately represented with DSCMs. An immediate consequence of this construction is a graphical Markov property for systems of SDEs. We define a \emph{time-splitting} operation, allowing us to analyse the concept of \emph{local independence} (a notion of continuous-time Granger (non-)causality). We also define a \emph{subsampling} operation, which returns a discrete-time DSCM, and which can be used for mathematical analysis of subsampled time-series. We give suggestions how DSCMs can be used for identification of the causal effect of time-dependent interventions, and how existing constraint-based causal discovery algorithms can be applied to time-series data.
\end{abstract}
% Keywords: stochastic differential equations, structural causal models, Markov property, local independence, causal effect identification, causal discovery
% TL;DR: We show that Dynamic Structural Causal Models are appropriate causal models of systems of stochastic differential equations, and investigate consequences of this result.

\red{A revised and extended version of this work has appeared as \citet{boeken2026causal}. It adds some nuances and corrects a number of points that are treated too quickly here; these are indicated by red footnotes throughout. All proofs are deferred to \cite{boeken2026causal}.}

\section{Introduction}
Many real-world systems exhibit a non-trivial dependence on time. In science and engineering, such systems are often modelled with Ordinary Differential Equations when the dynamics are deterministic, with Random Differential Equations (RDEs) when the dynamics of a population is described but every individual can have different (deterministic) dynamics, and with Stochastic Differential Equations (SDEs) when the dynamics of each individual are inherently stochastic \citep{banks2014modeling}. In this work, we will focus on the latter formalism, which generalises the former two. SDEs do not come formally equipped with causal semantics. As their mathematical formulation is rather technical, it is a nontrivial task to apply existing methods of causal reasoning to this model class. In this work, we investigate how SDEs can be mapped into a \emph{Dynamic Structural Causal Model} (DSCM), originally introduced by \cite{rubenstein2018deterministic} and refined in this work, to naturally bridge this gap. More precisely, we
\begin{enumerate}
	\item refine the definition of DSCMs to make them specific versions of SCMs as defined by \cite{bongers2021foundations}, enhancing them with a Markov property based on $\sigma$-separation;
	\item show how certain \emph{uniquely solvable} systems of SDEs can be transformed into DSCMs;
	\item define \emph{time-splitting} and \emph{subsampling} as formal tools to analyse the time-dependent structure of DSCMs;
	\item show how time-split DSCMs allow for a graphical interpretation of local independence and that the graph of `collapsed' DSCMs are local independence graphs, yielding a local independence Markov property;
	\item give further suggestions of how existing results for `static' causal effect identification and constraint-based causal discovery are naturally applicable to DSCMs.
\end{enumerate}

\section{Systems of Stochastic Differential Equations}\label{sec:sdes}
Consider the time-index $\Tcal := [0,T]$ for some $T\in \RR$. Let $D(\Tcal, \RR^n)$ be the space of càdlàg\footnote{continue \`a droite, limité \`a gauche} functions from $\Tcal$ to $\RR^n$. Throughout, we equip $D(\Tcal, \RR^n)$ with the $J_1$ topology \citep{skorokhod1956limit} and the Borel sigma-algebra, making it a standard Borel space.
Let $(\Omega, \FF, \PP)$ denote a filtered probability space with filtration $\FF := (\Fcal_t)_{t\in \Tcal}$ that is right-continuous and contains all $\PP$-null-sets. A process $X: \Omega \times \Tcal \to \RR^n$ is $\FF$-adapted if $X(t) \in \Fcal_t$ for all $t\in\Tcal$; for predictable processes we have $X(t) \in \Fcal_{t-}$, where $\Fcal_{t-} := \bigcap_{s<t}\Fcal_s$. Let $\mathbb{S}_\FF(\Tcal, \RR^n)$ denote the class of $\RR^n$-valued semimartingales, i.e.\ processes $X : \Omega \times \Tcal \to \RR^n$ that are $\FF$-adapted with càdlàg sample paths a.s., and with (unique) decomposition $X = A + M$ with $A$ a predictable process of finite variation and $M$ a local martingale. Let $\Pcal_b$ denote the set of locally bounded predictably measurable processes.
% For each $X\in \SSb_\FF$, the decomposition $X = M + A$ is unique.
% When $H$ and $X$ are semimartingales with no common discontinuities, $H$ of finite $p$-variation and $X$ of finite $q$-variation with $p^{-1} + q^{-1} > 1$, the integral $\int_0^t H(s)\diff X(s)$ can be defined analytically (pathwise) as the Riemann-Stieltjes integral \citep{lyons2007differential}. This for example holds when $H$ is continuous and $X$ has bounded (1-)variation. Note that Brownian motion $W$ only has finite $p$-variation for $p > 2$, preventing a pathwise construction of the stochastic integral $\int H\diff W$ when $H$ is not sufficiently regular. To overcome this hurdle, the underlying probability space of the stochastic processes is leveraged to still give a sensible meaning to a stochastic integral.
The \emph{stochastic integral of $G$ w.r.t.\ $H$}
\begin{equation}\label{eqn:stochastic_integral}
	J: \Pcal_b \times \SSb \to \SSb, ~~(G, H) \mapsto \int_0^t G(s)\diff H(s)
\end{equation}
is a well-defined notion (\cite{protter2005stochastic}, Thm. IV.15).

For a formal definition of a system of stochastic differential equations, we require the following notion:
\begin{assumption}\label{ass:cadlag_adapted}
	Let $f : \Tcal \times D(\Tcal, \RR^m) \to \RR^n$ be a Borel-measurable function such that
	\begin{enumerate}[label=\roman*)]
		\item $t\mapsto f(t, x) \in D(\Tcal, \RR^n)$ for all $x\in D(\Tcal, \RR^m)$,
		      \item\label{ass:cadlag_adapted:2} $f(t, x) = f(t, x^t)$ for all $x\in D(\Tcal, \RR^m)$ and $t\in\Tcal$, where $x^t(s) := x(s\wedge t)$.
	\end{enumerate}
\end{assumption}

\begin{remark}\label{rem:fn_of_sm_is_sm}
	If $f$ satisfies Assumption \ref{ass:cadlag_adapted} and $X \in \SSb_\FF(\Tcal, \RR^m)$, then the process $f(t, X) $ is in $\SSb_\FF(\Tcal, \RR^n)$, and $f(t-, X):= \lim_{s\uparrow t} f(t, X)$ is an element of $\Pcal_b$ \citep{przybylowicz2023skorohod}.
\end{remark}

\begin{definition}[System of SDEs, solutions]\label{def:sdes}
	Given a filtered probability space $(\Omega, \FF, \PP)$, a \emph{system of stochastic differential equations (SDEs)} is a tuple $\Dcal = \angs{V, W, (X_V^0, X_W), f}$ with
	\begin{enumerate}[label=\roman*)]
		\item $V$ and $W$ finite disjoint index sets,
		\item for each $v\in V$ a random variable $X_v^0 :\Omega \to \RR$, and for each $w\in W$ a stochastic process $X_w \in \SSb_\FF(\Tcal, \RR)$,
		\item for each $v\in V$ a stochastic differential equation
		      \begin{equation*}
			      f_v:  X_v(t) = X_v^0 + \int_0^t g_{v}(s-, X_{\alpha(v)})\diff h_v(s, X_{\beta(v)})
		      \end{equation*}
		      for all $t\in\Tcal$, for sets $\alpha(v), \beta(v) \subseteq V \cup W$, and functions $g_v: \Tcal \times D(\Tcal, \RR^{|\alpha(v)|}) \to \RR^{m_v}$ and $h_v : \Tcal \times D(\Tcal, \RR^{|\beta(v)|}) \to \RR^{m_v}$ for some $m_v \in \NN$, both satisfying Assumption \ref{ass:cadlag_adapted}.\footnote{We write $\int_0^t g\diff h$ for $\sum_{i=1}^{m}\int_0^t g_{i}\diff h_{i}$.}
	\end{enumerate}
	The processes $g_v(t, X_{\alpha(v)})$ and $h_v(t, X_{\beta(v)})$ are called \emph{integrands} and \emph{integrators} respectively. A \emph{solution}\footnote{Commonly referred to as a \emph{strong solution}.} of a system of SDEs $\Dcal$ is a stochastic process $X_V \in \SSb_\FF(\Tcal, \RR^{|V|})$ such that $X_V = f_V(X_V^0, X_V, X_W)$ holds $\PP$-a.s.
\end{definition}

By Remark \ref{rem:fn_of_sm_is_sm}, we have for semimartingales $X_{\alpha(v)}$ and $X_{\beta(v)}$ that $g_{v}(s-, X_{\alpha(v)})$ is in $\Pcal_b$ and $h_v(s, X_{\beta(v)})$ is in $\SSb$, making the stochastic integrals well-defined (see equation \ref{eqn:stochastic_integral}). Note that we typically have $v\in \alpha(v)$. Definition \ref{def:sdes} generalises common definitions of SDEs, which typically only model the relations between the variables that appear in the integrands, and assume the integrators to be given. We allow for explicit modelling of the variables that appear in the integrators because it is mathematically feasible, has applications in controlled SDEs \citep{lyons2007differential}, and allows for the modelling of instantaneous effects in time-split DSCMs (see Section \ref{sec:time-evaluations}).

\begin{remark}
	For specific choices of integrators and integrands, SDEs can model some special cases:
	\begin{enumerate}
		\item If $g(t, X_{\alpha}) = g(t, X_\alpha(t))$ and $h$ is a Lévy process independent of $X^0$, then any solution $X$ is temporally Markov\footnote{Process $X\in \SSb$ is called \emph{temporally Markov} if $\PP(X_{t+s} | \Fcal_{t}) = \PP(X_{t+s} | X_{t})$ for all $s,t\in \Tcal$ such that $s+t \in \Tcal$; it is (temporally) \emph{strong Markov} if this holds for any stopping time $t$.} (\citealp{protter2005stochastic}, Theorem V.6.32). If additionally $g(t, X_{\alpha}) = g(X_\alpha(t))$, then the solution is strong Markov; such dynamics are called \emph{time-invariant}.
		\item If $h(s) = s$, the `stochastic integral' $\int g(s, X_{\alpha})\diff h(s)$ reduces to the Riemann integral $\int g(s, X_\alpha)\diff s$, so a system of Random Differential Equations is a special case of a system of SDEs.
		\item If $h(t, W) = W(t)$ is a Brownian Motion, any solution $X$ has continuous and possibly non-differentiable sample paths. If $h(t, W) = (t, W(t))$ and the dynamics are time-invariant, such an SDE is called an \emph{It\^o diffusion}.
		\item If $h(t, N) = N(t)$ is a jump process (e.g.\ a Poisson process), the solution $X$ can have jumps as well. If $h(t, W, N) = (t, W(t), N(t))$ and the dynamics are time-invariant, such an SDE is called a \emph{jump-diffusion}.
		\item If $H$ is a process with differentiable sample paths, an RDE $X = \int g(H'(s), X(s))\diff s$ with $g$ linear in its first argument can be suitably modelled with an SDE $X = \int f(X)\diff H(s)$ \citep{lyons2007differential}. The SDE formulation generalises this concept to a large class of non-differentiable processes $H$.
		\item The functional relationship $X(t) = h(t, X_{\beta})$ can be equivalently expressed as $X(t) = h(0, X_{\beta})+ \int_0^t \diff h(s, X_{\beta})$.
	\end{enumerate}
\end{remark}

\begin{example}\label{ex:sde}
	Consider the following system of Stochastic Differential Equations
	\begin{align*}
		\Dcal: \begin{cases}
			       X_1(t) = X_1^0 + \int_0^t g_1(s-, X_1, X_3) \diff W(s) \\
			       X_2(t) = X_2^0 + \int_0^t g_2(s-, X_1, X_2) \diff N(s) \\
			       X_3(t) = X_3^0 + \int_0^t g_3(s-, X_2, X_3) \diff W(s) \\
			       X_4(t) = X_4^0 + \int_0^t g_4(s-, X_4) \diff X_2(s)
		       \end{cases}
	\end{align*}
	with $W$ a Brownian motion, $N$ a Poisson process (independent from $W$), and $X_1^0, X_2^0, X_3^0, X_4^0$ independent random variables. Note that $X_2$ is the integrator of $X_4$.
\end{example}

To ensure existence and uniqueness of solutions of a system of SDEs, we require additional regularity of the function $g$:
\begin{assumption}\label{ass:cadlag_adapted_lip}
	Let $g : \Tcal \times D(\Tcal, \RR^m) \to \RR^n$ satisfy Assumption \ref{ass:cadlag_adapted}, and let there exist a $K\in D(\Tcal, (0, \infty))$ such that for all $x, x_1, x_2 \in D(\Tcal, \RR^m)$ we have
	
	\begin{align*}
		|g(t, x)|               & \leq K(t)\left(1+ \sup_{0\leq s \leq t}|x(s)|\right) \\
		|g(t, x_2) - g(t, x_1)| & \leq K(t)\sup_{0\leq s\leq t}|x_2(s) - x_1(s)|.
	\end{align*}
\end{assumption}

\begin{theorem}[\cite{przybylowicz2023skorohod}]\label{thm:ito_map}
	For given functions $g: \Tcal \times D(\Tcal, \RR^{k+1}) \to \RR^{m}$ and $h : \Tcal \times D(\Tcal, \RR^{\ell}) \to \RR^{m}$ that satisfy Assumptions \ref{ass:cadlag_adapted_lip} and \ref{ass:cadlag_adapted} respectively, there exists a Skorohod measurable mapping
	\begin{equation*}
		I : \RR \times D(\Tcal, \RR^{k}) \times D(\Tcal, \RR^{\ell}) \to D(\Tcal, \RR)
	\end{equation*}
	such that for any filtered probability space $(\Omega, \FF, \PP)$ with square-integrable random variable $X^0:\Omega \to \RR$ and semimartingales $A \in \SSb_\FF(\Tcal, \RR^k), B\in\SSb_\FF(\Tcal, \RR^\ell)$, the process $X := I(X^0, A, B) \in \SSb_\FF(\Tcal, \RR)$ is the unique solution of
	\begin{equation*}
		X(t) = X^0 + \int_0^t g(s-, X, A)\diff h(s, B).
	\end{equation*}
\end{theorem}

We refer to such a solution function $I$ as an \emph{It\^o map}. The above result straightforwardly extends to processes $X$ taking values in $\RR^n$.

\subsection{Towards causal reasoning with SDEs}
% Include definition of interventions on SDEs here
To analyse whether a system of SDEs can be solved, we use the dependency structure between the processes, which can be suitably depicted with a graph:
\begin{definition}[Graph of system of SDEs]
	For a system of SDEs $\Dcal = \angs{V, W, (X_V^0, X_W), f}$, let $G(\Dcal) = (V\cup W, E)$ be the directed graph where $E := \{v \to w : w \in V, v \in \alpha(w)\cup \beta(w) \setminus \{w\}\}$.
\end{definition}

\begin{example}
	Let $\Dcal$ be the systems of SDEs from Example \ref{ex:sde}. The graph $G(\Dcal)$ is given in Figure \ref{fig:sde_graph}. Outgoing edges of integrators are displayed in red.
\end{example}
\begin{figure}[!htb]
	\centering
	\begin{tikzpicture}[scale=1,xscale=1.1,yscale=1]
		\node[var] (X1) at (0, 0) {$X_1$};
		\node[var] (X3) at (0.5, -1) {$X_3$};
		\node[var] (X2) at (-.5, -1) {$X_2$};
		\node[var] (X4) at (0, -2) {$X_4$};
		\node[var] (xi1) at (-1.5, 0) {$X^0_1$};
		\node[var] (xi2) at (-1.5, -1) {$X^0_2$};
		\node[var] (xi3) at (1.5, -1) {$X^0_3$};
		\node[var] (xi4) at (-1.5, -2) {$X^0_4$};
		\node[var] (W1) at (1.5, 0) {$W$};
		\node[var] (N3) at (1.5, -2) {$N$};
		\draw[arr] (X1) to (X2);
		\draw[arr] (X2) to (X3);
		\draw[arr] (X3) to (X1);
		\draw[arr,color=red] (X2) to (X4);
		\draw[arr] (xi1) to (X1);
		\draw[arr] (xi2) to (X2);
		\draw[arr] (xi3) to (X3);
		\draw[arr] (xi4) to (X4);
		\draw[arr,color=red] (W1) to (X1);
		\draw[arr,color=red] (W1) to (X3);
		\draw[arr,color=red] (N3) to (X2);
	\end{tikzpicture}
	\caption{The directed graph $G(\Dcal)$.}
	\label{fig:sde_graph}
\end{figure}

By partitioning the set of endogenous processes $V$ into \emph{strongly connected components} (SCCs) (subsets whose every two vertices are connected by a directed path) and assuming that no element of such a SCC is the integrator of another element in the SCC (so no directed cycle contains a red edge), we can iteratively apply Theorem \ref{thm:ito_map} along the topological order of the SCCs to obtain a unique solution.

\begin{proposition}[Unique solvability of a system of SDEs]\label{thm:solvable_sde}
	Let $\Dcal = \angs{V, W, (X_V^0, X_W), f}$ be a system of SDEs such that for every SDE $f_v$, the parameters $g_v$ and $h_v$ satisfy Assumptions \ref{ass:cadlag_adapted_lip} and \ref{ass:cadlag_adapted} respectively. If for every $v\in V$ we have $\beta(v) \cap \Sc(v) = \emptyset$,\footnote{$\Sc(v)$ denotes the strongly connected component of $v$.} there exists an It\^o map $I_V$ such that $X_V := I_V(X_V^0, X_W) \in \SSb_\FF(\Tcal, \RR^{|V|})$ is the unique solution of $\Dcal$.
\end{proposition}

We refer to a system of SDEs $\Dcal$ that satisfies Proposition \ref{thm:solvable_sde} as \emph{uniquely solvable}.

\begin{remark}
	A system of SDEs is assumed to have a certain causal interpretation: for the SDEs in Definition \ref{def:sdes}, it is assumed that the left-hand side is \emph{determined} by variables on the right-hand side. In Section \ref{sec:sde_to_dscm} we will indeed interpret the It\^o map $I_v$ for and SDE $f_v$ as a \emph{structural equation} of a DSCM. This construction then gives a clear causal interpretation to the graph $G(\Dcal)$, where the processes $V$ and $W$ are interpreted as endogenous and exogenous respectively. Compared to existing work where for a \emph{systems of constraints} one deduces the causal structure (e.g.\ \citealp{iwasaki1994causality,blom2021conditional}), our setting makes relatively strong assumptions.
\end{remark}

Another important causal aspect that is engrained in the system of SDEs is whether processes exhibit instantaneous dependencies. Although not explicitly stated in \cite{przybylowicz2023skorohod}, one can verify via their proof of Theorem \ref{thm:ito_map} that the It\^o map $I$ satisfies $I(x^0, a, b)^t = I(x^0, a^{t-}, b^t)^t$. As it will play an important role in the analysis of temporal dynamics of DSCMs, we formally define this concept:

\begin{definition}[Adapted, predictable dependence]\label{def:adapted_predictable}
	Let $g : D(\Tcal, \RR^m) \to D(\Tcal, \RR^n)$ be a Borel-measurable function. We say that the dependence of $g$ on $x$ is \emph{adapted} if $g(x)^t = g(x^t)^t$ and \emph{predictable} if $g(x)^t = g(x^{t-})^t$ for all $t\in \Tcal$.
\end{definition}

\begin{remark}\label{rem:sol_fn_adaptedness}
	For given process $X$ and $Y:=g(X)$ for measurable $g$, we have $Y \in \Fcal_t^X$ if $g$ is adapted, and $Y\in \Fcal_{t-}^X$ if $g$ is predictable,\footnote{Here, $\Fcal_t^X$ denotes the filtration generated by $X$.} motivating the nomenclature \emph{adapted} and \emph{predictable}.
	For any uniquely solvable system of SDEs $\Dcal$ and any $v\in V$, the process $X_v$ has predictable dependence on $X_{\alpha(v)}$ and adapted (and thus possibly instantaneous) dependence on $X_{\beta(v)}$. Graphically, we depict adapted dependencies with red edges.
\end{remark}

\section{Dynamic Structural Causal Models}\label{sec:dscms}
Following \cite{bongers2021foundations} and \cite{forre2023mathematical}, a \emph{Structural Causal Model} (SCM) is a tuple $\Mcal = \angs{V, W, \Xcal_V, \Xcal_W, f, \PP(X_W)}$ where $V, W$ are disjoint finite index sets of \emph{endogenous variables} and \emph{exogenous variables} respectively, the \emph{domains} $\Xcal_V = \prod_{i\in V}\Xcal_i$ and $\Xcal_W = \prod_{i\in W}\Xcal_i$ are products of \emph{standard Borel spaces} $\Xcal_i$, the \emph{exogenous distribution} $\PP(X_W) = \bigotimes_{w\in W}\PP(X_w)$ is a product of probability distributions, and the \emph{causal mechanism} $f : \Xcal \to \Xcal_V$ is a measurable function. Typically, SCMs are used to model `static' systems, where variables take values in $\RR$, for example. Discrete-time dynamical systems are often modelled with a copy of all variables $X_V$ for every time-step; an approach that does not directly apply to continuous-time dynamical systems. Inspired by \citep{rubenstein2018deterministic} and \cite{bongers2022causal}, we define a specific type of SCM where variables are functions on some time-interval, and the structural equations respect temporal causality:

\begin{definition}[DSCM]
	Given a time-interval $\Tcal = [0,T]$ for some $T\in \RR_+$ or $\Tcal \subseteq \{1, ..., T\}$ for some $T\in \NN$, a \emph{Dynamic Structural Causal Model} is an SCM $\Mcal = \angs{V, W, \Xcal_V, \Xcal_W, f, \PP}$ with
	\begin{enumerate}[label=\roman*)]
		\item $V := V_p \times \Ecal$ with endogenous processes $V_p$ and $\Ecal$ a set of disjoint \emph{evaluation intervals} or \emph{points} in $\Tcal$;\footnote{For $(v, \Ical) \in V\times \Ecal$ we write $v^{\Ical} := (v, \Ical)$ and $X_v^{\Ical} := X_{(v, \Ical)}$ and $X_v := X_v^{\Tcal}$.}
		\item exogenous processes $W$;
		\item standard Borel spaces $\Xcal_V = \bigtimes_{(v, \Ical)\in V\times \Ecal} D(\Ical, \RR)$ and $\Xcal_W = \bigtimes_{w \in W} D(\Tcal, \RR)$;
		\item measurable structural equations $f: \Xcal_V \times \Xcal_W \to \Xcal_V$ that are \emph{adapted} (see Definition \ref{def:adapted_predictable});
		\item exogenous distribution $\PP(X_W) = \bigotimes_{w\in W}\PP(X_w)$.
	\end{enumerate}
\end{definition}
The set $\Ecal$ denotes the time intervals (or points) on which we evaluate the processes $X_{V_p}$. For example, when $\Ecal = \{\{0\}, (0,T]\}$, for every $v\in V_p$ we will consider the variables $X_v^0$ and $X_v^{(0,T]}$. We refer to a DSCM as \emph{collapsed} if $\Ecal = \Tcal$. For $v\in V_p$, we denote with $\alpha(v)$ and $\beta(v)$ the sets of parents of $v$ on which it has predictable and adapted dependence respectively.

For notational simplicity, endogenous parameters like initial values or parameters that `materialise' at some $t_0 \in \Tcal$ (i.e.\ other variables are only allowed to depend on it from $t_0$ onwards) are modelled as elements of $D(\{t_0\}, \RR)$. For a function $x:(s, t)\to \RR$ we use the convention that $x(r) = x(s+)$ for all $r < s$ and $x(u) = x(t-)$ for all $u > t$.\footnote{Here, $x(t+)$ denotes $\lim_{s \downarrow t} x(s)$.}

% Although generally defined, we will mainly focus on DSCM induced by systems of SDEs $\Dcal = \angs{V_p, W, (X_V^0, X_W), f}$ living on $\Tcal = [0,T]$, mapped to a with set of processes $V_p$ and evaluations $\Ecal = \{\{0\}, (0,T]\}$, giving rise to the initial values $X_V^0$ and processes $X_V^{(0,T]}$. This construction will be given in Section \ref{sec:sde_to_dscm}. In Section \ref{sec:time-evaluations} we will consider a \emph{time-split} DSCM $\Mcal_{\ev^+(t)}$ with variables $X_V^0, X_V^{(0,t)}, X_V^t, X_V^{(t,T]}$, and \emph{subsampled} DSCM $\Mcal_{\ev(s, t)}$ with variables $X_V^0, X_V^{s}, X_V^t$.

We have refined the definition of DSCMs as given by \cite{rubenstein2018deterministic} by 1) ensuring the structural equations respect temporal causality, 2) allowing for stochasticity through the exogenous variables, and 3) picking a standard Borel space ($D(\Tcal, \RR)$) as space of trajectories, making for all $A, B\subseteq V\cup W$ the conditional distribution $\PP(X_A \mid X_B)$ well-defined, and with that giving a well-defined notion of conditional independence.
% Additionally, Section \ref{sec:sde_to_dscm} generalises their mapping of differential equations to DSCMs from ODEs to SDEs.
DSCMs differ from \emph{Structural Dynamical Causal Models} (SDCMs) \citep{bongers2022causal} as we model entire sample paths as a whole, instead of separately at every time $t\in\Tcal$. When $\Tcal$ or $\bigcup \Ecal$ is a discrete set, the DSCM reduces to a structural vector autoregression model \citep{peters2013causal,malinsky2018causala} with possible serial dependence of the noise variables, or a Dynamic Bayesian Network \citep{dean1989model}.

Perfect interventions on DSCMs are defined exactly the same as for SCMs:
\begin{definition}[Perfect interventions on DSCMs]
	Given a DSCM $\Mcal$, intervention target $T\subseteq V$ and intervention value $x_T\in \Xcal_T$, the perfectly intervened DSCM is defined as $\Mcal_{\Do(X_T=x_T)}:= \angs{V, W, \Xcal_V, \Xcal_W, f^\circ, \PP(X_{W})}$ where the components of the intervened causal mechanism $f^\circ : \Xcal_V\times \Xcal_W \to \Xcal_{V}$ are given by
	\begin{equation*}
		f^\circ_v(x_V, x_W) := \begin{cases}
			f_v(x_V, x_W) & \text{if $v \in V\setminus T$} \\
			x_v           & \text{if $v \in T$}.
		\end{cases}
	\end{equation*}
\end{definition}

For a given SCM $\Mcal$, a random variable $X_V \in \Xcal_V$ is a \emph{solution} of the SCM if $X_V = f(X_V, X_W)$ holds $\PP(X_W)$-a.s. Existence and uniqueness of solutions are generally not guaranteed for (intervened) cyclic SCMs. To this end, \cite{bongers2022causal} introduced the following class of SCMs:

\begin{definition}[Simple (D)SCM]\label{def:simple_dscm}
	A (D)SCM is simple if for all $O\subseteq V$ (write $T := V\setminus O$) there exists a measurable function $g_O:\Xcal_{T} \times \Xcal_W\to \Xcal_O$ such that for $\PP(X_W)$-almost all $x_W \in \Xcal_W$, for all $x_V \in \Xcal_V$ we have $x_O = g_O(x_{T}, x_W)$ if and only if $x_O = f_O(x_O, x_T, x_W)$.
\end{definition}

For a simple DSCM $\Mcal$, the random variable $g_O(x_T,X_W)$ is a solution for $\Mcal_{\Do(X_T = x_T)}$, implying the existence of a well-defined \emph{interventional distribution} $\PP(X_V|\Do(X_T = x_T))$. The function $g_O$ is referred to as a \emph{solution function} for $\Mcal_{\Do(X_T=x_T)}$. The following proposition shows that solutions of DSCMs obey temporal causality:

\begin{proposition}\label{thm:solution_function_adapted}
	Given a simple DSCM $\Mcal$, the solution function $g_{O}(x_T,x_W)$ of the intervened DSCM $\Mcal_{\Do(X_T = x_T)}$ is adapted.\notefn{Simplicity does not imply adaptedness of the solution function. A sufficient condition is that for every $0\leq t\leq T$ the DSCM restricted to $[0,t]$ is simple.}
\end{proposition}

\begin{remark}
	Following Remark \ref{rem:sol_fn_adaptedness}, adaptedness of the solution function $g_{O}(x_T,x_W)$ is equivalent to adaptedness of solution process $X_V := g_O(x_T, X_W)$ to a canonical filtration on the underlying probability space $\Xcal_{T}\times \Xcal_W$.
\end{remark}

In line with marginalisation of probability distributions, the marginalisation of a simple DSCM is defined as follows:
\begin{definition}[Marginalisation]
	For simple DSCM $\Mcal$ and subset $L \subseteq V$ (write $O := V\setminus L$), the marginalised DSCM $\Mcal_{\marg(L)} := \angs{O, W, \Xcal_{O}, \Xcal_W, \tilde{f}, \PP(X_W)}$ has marginal causal mechanism
	\begin{equation*}
		\tilde{f}(x_O, x_W) := f_O(x_O, g_L(x_{O}, x_W), x_W),
	\end{equation*}
	where $g_L$ is the solution function in $\Mcal_{\Do(X_O = x_O)}$.
\end{definition}

\begin{proposition}\label{thm:marginalised_simple_dscm_is_simple}
	The class of simple DSCMs is closed under marginalisation.\notefn{This does not hold for the class of simple DSCMs as stated, as the proof relies on adaptedness of the solution function (Proposition \ref{thm:solution_function_adapted}). It does hold both under the stronger condition that the DSCM restricted to $[0,t]$ is simple for every $0\leq t\leq T$, and under a much weaker notion of unique solvability that quite a general class of SDE-induced DSCMs satisfies \citep{boeken2026causal}.}
\end{proposition}

\cite{bongers2021foundations} show that marginalised simple SCMs are observationally, interventionally and counterfactually equivalent to the original SCM with respect to the observed variables. This makes marginalisation a powerful notion of abstraction, as any causal inference on the observed part of a system remains valid in the underlying system.

\subsection{A Markov property for simple DSCMs}
For SCM $\Mcal$ and nodes $i \in V\cup W$ and $j\in V$ we call $i$ a \emph{parent} of $j$ if there does not exist a measurable function $\tilde{f}_j : \Xcal_{V\cup W \setminus \{i\}}\to \Xcal_j$ such that $f_j(x_V, x_W) = \tilde{f}_j(x_{V\cup W \setminus \{i\}})$ for $\PP(X_W)$-almost all $x_W\in\Xcal_W$ and for all $x_V\in \Xcal_V$. We call a \emph{directed graph} $G^+(\Mcal) := (V\cup W, E)$ with directed edges $E = \{i\to j : i\in V\cup W, j\in V, \text{$i$ is a parent of $j$ in $\Mcal$}\}$ the \emph{augmented graph} of $\Mcal$. Often we don't display the exogenous nodes in the graph, in which case we consider the \emph{directed mixed graph} $G(\Mcal)=(V, E')$ where $E' = E\cup \{i \leftrightarrow j : i, j \in V, (i \leftarrow k \to j) \in G(\Mcal) \text{ for some $k\in W$}\}$, where $E$ stems from $G^+(\Mcal)$. We call an SCM \emph{cyclic} if its graph $G(\Mcal)$ has a directed cycle. Simple SCMs can be cyclic, but cannot have self-cycles.
A suitable extension of $d$-separation \citep{pearl2009causality} from DAGs to DMGs is the following notion of $\sigma$-separation:

\begin{definition}[$\sigma$-separation, \cite{forre2017markov,bongers2021foundations}]
	Given a DMG $G=(V, E)$, a set of nodes $C\subset V$ and a walk $\pi$ in DMG $G$:
	\begin{enumerate}
		\item a non-collider is called \emph{blockable} if it has a child on the walk that is not in the same strongly connected component,
		\item the walk $\pi$ is called \emph{$\sigma$-blocked by $C$} if there is a collider on $\pi$ that is not in $\Anc(C)$, or if there is a blockable non-collider on $\pi$ in $C$.
	\end{enumerate}
	For sets of nodes $A, B, C\subseteq V$, we call $A$ and $B$ \emph{$\sigma$-separated} given $C$, written $A\perp_G^\sigma B \given C$, if all paths between $A$ and $B$ are $\sigma$-blocked by $C$.
\end{definition}

In general $\sigma$-separation implies $d$-separation, and for acyclic DMGs the two notions coincide.

\begin{example}
	In the DMG $G$ from Figure \ref{fig:sde_graph} we have $X_1^0 \Perp^d_G X_2^0 \mid X_1, X_2$, and $X_1^0 \nPerp^\sigma_G X_2^0 \mid X_1, X_2$.
\end{example}

\begin{theorem}[Markov Property, \cite{forre2017markov,bongers2021foundations}]\label{thm:markov_property}
	For simple (D)SCM $\Mcal$ with distribution $\PP(X_V, X_W)$, graph $G(\Mcal)$ and (not necessarily disjoint) sets $A, B, C \subseteq V\cup W$, we have\notefn{For a class of (possibly cyclic) additive-noise SDEs driven by Brownian motion, \citet{boeken2026causal} strengthen Theorem \ref{thm:markov_property} to a $d$-separation Markov property. For discrete-time systems without instantaneous cycles, the $d$-separation Markov property follows from \cite{ferreira2024identifying}.}
	\begin{equation*}
		A \perp_{G(\Mcal)}^\sigma B \given C \implies X_A \Indep_{\PP} X_B \given X_C.
	\end{equation*}
\end{theorem}
A simple (D)SCM $\Mcal$ is called $\sigma$-\emph{faithful} if the reverse implication holds; this need not hold in general.

\section{Constructing a DSCM from a system of SDEs}\label{sec:sde_to_dscm}
As mentioned before, one of our goals is to map a system of SDEs $\Dcal$ to a (D)SCM, so we can apply all tools that are available for SCMs to SDEs. However, an SDE cannot directly be interpreted as structural equation of a DSCM, as SDEs are not equations between variables taking values in $D(\Tcal, \RR)$, but in $\SSb(\Tcal, \RR)$.
% This complicates matters, as usually in the SCM framework the distribution over the endogenous variables is obtained by taking the pushforward of the exogenous distribution through the solution functions of the SCM. If the variables are stochastic processes to begin with, we can't ensure the existence of a joint distribution over all processes (e.g. when SDEs only have \emph{weak solutions}). Then you can't do conditioning, and hence we can't rely on the usual notion of conditional independence.
However, when $\Dcal$ is uniquely solvable, we \emph{can} unambiguously interpret the It\^o maps of $\Dcal$ as the structural equations of a DSCM $\Mcal_\Dcal$.

\begin{definition}[DSCM induced by a system of SDEs]
	Given a uniquely solvable system of SDEs $\Dcal = \angs{V_p, W_p, (X_{V_p}^0, X_{W_p}), f}$ on a filtered probability space $(\Omega, \FF, \PP)$ and time-interval $[0,T]$, we define \emph{the simple DSCM induced by $\Dcal$} as $\Mcal_\Dcal := \angs{V, W, \Xcal_V, \Xcal_W, f^*, \PP(X_W)}$ with
	\begin{enumerate}
		\item $V := V_p \times \Ecal$  with $\Ecal := \Tcal$;
		\item $W := W_0 \cup W_p$ with $W_0 := \{v^0 : v\in V_p\}$ a copy of $V_p$ as the indices of the initial values;
		\item for every $v\in V_p$ the structural equation
		      \begin{equation*}
			      f^*_{v}(x_V, x_W) := I_{v}(x_{v}^0, x_{\alpha(v)}, x_{\beta(v)})
		      \end{equation*}
		      where $I_v$ is the It\^o map for the SDE $f_v$;
		\item $\Xcal_V := D(\Tcal, \RR^{|V|})$ and $\Xcal_W := D(\Tcal, \RR^{|W|})$;
		\item $\PP(X_W)$ is the law of $(X_{V_p}^0, X_{W_p})$.
	\end{enumerate}
\end{definition}
Note that $G^+(\Mcal_\Dcal) = G(\Dcal)$, and does not have any self-loops $v\to v$. Interventions on a SDE $\Dcal$ can be defined similarly as for DSCMs,\notefn{The It\^o map of Theorem \ref{thm:ito_map} only represents the stochastic integral for every \emph{semimartingale} integrator. When intervening on the integrator and thus setting it to a deterministic path, this path needs to be of finite variation for it to be a semimartingale. Perfect interventions on endogenous integrators should therefore be restricted to intervention values of finite variation.} intervening and mapping to a DSCM commute, and $\Mcal_\Dcal$ is indeed simple.\notefn{Simplicity of $\Mcal_\Dcal$ does not follow as claimed, since Theorem~\ref{thm:ito_map} yields, for each adapted input process $X_T$ (in the notation of Definition~\ref{def:simple_dscm}), a solution that is unique up to a null set that may depend on $X_T$, whereas simplicity in the sense of \cite{bongers2021foundations} requires the fixed-point and uniqueness properties to hold on a single null set that is valid for all values $x_T$. Hence, the $\sigma$-separation Markov property and do-calculus do not follow automatically for $\Mcal_\Dcal$, and neither does closure of this class under marginalisation. Nevertheless those results hold, as shown by \cite{boeken2026causal}.}

\begin{example}\label{ex:dscm}
	Let $\Dcal$ be the systems of SDEs from Example \ref{ex:sde}. The structural equations of the induced DSCM $\Mcal_\Dcal$ are given by:
	
	\begin{equation*}
		\Mcal_\Dcal : \begin{cases}
			X_1 = f_1(X^0_1, X_3, W) \\
			X_2 = f_2(X^0_2, X_1, N) \\
			X_3 = f_3(X^0_3, X_2, W) \\
			X_4 = f_4(X^0_4, X_2).
		\end{cases}
	\end{equation*}
	The augmented graph $G^+(\Mcal_\Dcal) = G(\Dcal)$ is given in Figure \ref{fig:sde_graph}, and $G(\Mcal_\Dcal)$ is given in Figure \ref{fig:dscm_graph}, where the bidirected edge represents the latent confounder $W$.
\end{example}
\begin{figure}[!htb]
	\centering
	\begin{tikzpicture}[scale=1,xscale=1.1,yscale=1]
		\node[var] (X1) at (0, 0) {$X_1$};
		\node[var] (X3) at (0.5, -1) {$X_3$};
		\node[var] (X2) at (-.5, -1) {$X_2$};
		\node[var] (X4) at (0, -2) {$X_4$};
		% \node[var] (xi1) at (-1.5, 0) {$X^0_1$};
		% \node[var] (xi2) at (-1.5, -1) {$X^0_2$};
		% \node[var] (xi3) at (1.5, -1) {$X^0_3$};
		% \node[var] (xi4) at (-1.5, -2) {$X^0_4$};
		% \node[var] (W1) at (1.5, 0) {$W$};
		% \node[var] (N3) at (1.5, -2) {$N$};
		\draw[arr] (X1) to (X2);
		\draw[arr] (X2) to (X3);
		\draw[arr] (X3) to (X1);
		\draw[arr,color=red] (X2) to (X4);
		% \draw[arr] (xi1) to (X1);
		% \draw[arr] (xi2) to (X2);
		% \draw[arr] (xi3) to (X3);
		% \draw[arr] (xi4) to (X4);
		% \draw[arr] (W1) to (X1);
		\draw[biarr,bend left,color=red] (X1) to (X3);
		% \draw[arr] (N3) to (X2);
	\end{tikzpicture}
	\caption{The directed mixed graph $G(\Mcal_\Dcal)$.}
	\label{fig:dscm_graph}
\end{figure}

\begin{remark}
	A technical but important detail is that the It\^o maps $f^*_v$ (and thus the structural equations of $\Mcal_\Dcal$) are well-defined \emph{everywhere} on their domain, due to the result of \cite{przybylowicz2023skorohod}. If one obtains such an It\^o map through the factorisation lemma (\cite{kallenberg2021foundations}, Lemma 1.14) it is not everywhere uniquely defined, rendering the effects of interventions outside this essential support ambiguous. It is also important to note that the It\^o maps have the right measurability properties (with respect to the Borel sets generated by the $J_1$ topology).
\end{remark}

\begin{remark}
	The Markov property for DSCMs (Theorem \ref{thm:markov_property}) is closely related to the Markov property for the SDCM framework for RDEs \citep{bongers2022causal}. We conjecture that for a given system $\Dcal$ of uniquely solvable RDEs, the graph $G(\Mcal_\Dcal)$ of DSCM $\Mcal_\Dcal$ is a subgraph of the graph $G(\Rcal_\Dcal)$ of the SDCM $\Rcal_\Dcal$, which would make the DSCM Markov property at least as powerful as the SDCM Markov property.
\end{remark}

\section{Time-evaluations of DSCMs}\label{sec:time-evaluations}
\begin{figure*}[!htb]
	\vspace{-10pt}
	\centering
	\begin{tikzpicture}[scale=1,xscale=1.3,yscale=1.2]
		\node[] () at (-.75, 1.8) {$G(\Mcal)$:};
		\node[] (W1) at (-1, 1) {$W$};
		\node[] (W2) at (-1, -2) {$N$};
		\node[] (X1) at (0, 0) {$X_1^{[0,T]}$};
		\node[] (X2) at (0, -1) {$X_2^{[0,T]}$};
		\node[] (X3) at (0, -2) {$X_4^{[0,T]}$};
		% \node[] (X1[0]) at (-1, 0) {$X^0_1$};
		% \node[] (X2[0]) at (-1, -1) {$X^0_2$};
		% \draw[arr] (X1[0]) to (X1);
		% \draw[arr] (X2[0]) to (X2);
		\draw[arr,color=red] (W1) to (X1);
		\draw[arr,color=red] (W2) to (X2);
		\draw[arr,bend right=20] (X1) to (X2);
		\draw[arr,bend right=20] (X2) to (X1);
		\draw[arr,color=red] (X2) to (X3);
	\end{tikzpicture}
	\hspace{8pt}
	\begin{tikzpicture}[scale=1,xscale=1.3,yscale=1.2]
		\node[] () at (-2, -.5) {$\overset{\textrm{Def.\ \ref{def:time-split_dscm}}}{\implies}$};
		\node[] () at (-2, -1) {$\overset{\textrm{Def.\ \ref{def:collapsed_dscm}}}{\impliedby}$};
		\node[] () at (-.35, 1.8) {$G(\Mcal_{\ev^+(0,s,t)})$:};
		\node[] (W1) at (0, 1) {$W$};
		\node[] (N2) at (4, 1) {$N$};
		\node[] (X2[0s]) at (0, -1) {$X_2^{(0,s)}$};
		\node[] (X2[s]) at (1, -1) {$X_2^{s}$};
		\node[] (X2[st]) at (2, -1) {$X_2^{(s,t)}$};
		\node[] (X2[t]) at (3, -1) {$X_2^{t}$};
		\node[] (X2[tT]) at (4, -1) {$X_2^{(t,T]}$};
		\draw[arr,color=red!20,draw] (N2) to (X2[0s]);
		\draw[arr,color=red!20,draw] (N2) to (X2[s]);
		\draw[arr,color=red!20,draw] (N2) to (X2[st]);
		\draw[arr,color=red!20,draw] (N2) to (X2[t]);
		\draw[arr,color=red!20,draw] (N2) to (X2[tT]);
		\node[] (X1[0s]) at (0, 0) {$X_1^{(0,s)}$};
		\node[] (X1[s]) at (1, 0) {$X_1^{s}$};
		\node[] (X1[st]) at (2, 0) {$X_1^{(s,t)}$};
		\node[] (X1[t]) at (3, 0) {$X_1^{t}$};
		\node[] (X1[tT]) at (4, 0) {$X_1^{(t,T]}$};
		\node[] (X4[0s]) at (0, -2) {$X_4^{(0,s)}$};
		\node[] (X4[s]) at (1, -2) {$X_4^{s}$};
		\node[] (X4[st]) at (2, -2) {$X_4^{(s,t)}$};
		\node[] (X4[t]) at (3, -2) {$X_4^{t}$};
		\node[] (X4[tT]) at (4, -2) {$X_4^{(t,T]}$};
		\node[] (X1[0]) at (-1, 0) {$X^0_1$};
		\node[] (X2[0]) at (-1, -1) {$X^0_2$};
		\node[] (X4[0]) at (-1, -2) {$X^0_4$};
		\draw[arr] (X1[0]) to (X1[0s]);
		\draw[arr] (X2[0]) to (X2[0s]);
		\draw[arr,color=red!20,draw] (W1) to (X1[0s]);
		\draw[arr,color=red!20,draw] (W1) to (X1[s]);
		\draw[arr,color=red!20,draw] (W1) to (X1[st]);
		\draw[arr,color=red!20,draw] (W1) to (X1[t]);
		\draw[arr,color=red!20,draw] (W1) to (X1[tT]);
		\draw[arr] (X1[0s]) to (X1[s]);
		\draw[arr] (X1[s]) to (X1[st]);
		\draw[arr] (X1[st]) to (X1[t]);
		\draw[arr] (X1[t]) to (X1[tT]);
		\draw[arr] (X2[0s]) to (X2[s]);
		\draw[arr] (X2[s]) to (X2[st]);
		\draw[arr] (X2[st]) to (X2[t]);
		\draw[arr] (X2[t]) to (X2[tT]);
		\draw[arr] (X1[0]) to (X2[0s]);
		\draw[arr] (X2[0]) to (X1[0s]);
		\draw[arr] (X4[0]) to (X4[0s]);
		\draw[arr,bend right=20] (X1[0s]) to (X2[0s]);
		\draw[arr,bend right=20] (X2[0s]) to (X1[0s]);
		\draw[arr] (X1[0s]) to (X2[s]);
		\draw[arr] (X2[0s]) to (X1[s]);
		\draw[arr] (X1[s]) to (X2[st]);
		\draw[arr] (X2[s]) to (X1[st]);
		\draw[arr,bend right=20] (X1[st]) to (X2[st]);
		\draw[arr,bend right=20] (X2[st]) to (X1[st]);
		\draw[arr] (X1[st]) to (X2[t]);
		\draw[arr] (X2[st]) to (X1[t]);
		\draw[arr] (X1[t]) to (X2[tT]);
		\draw[arr] (X2[t]) to (X1[tT]);
		\draw[arr,bend right=20] (X1[tT]) to (X2[tT]);
		\draw[arr,bend right=20] (X2[tT]) to (X1[tT]);
		\draw[arr,color=red] (X2[0s]) to (X4[0s]);
		\draw[arr] (X2[0s]) to (X4[s]);
		\draw[arr,color=red] (X2[st]) to (X4[st]);
		\draw[arr] (X2[st]) to (X4[t]);
		\draw[arr,color=red] (X2[tT]) to (X4[tT]);
		\draw[arr,color=red] (X2[s]) to (X4[s]);
		\draw[arr,color=red] (X2[t]) to (X4[t]);
		\draw[arr] (X2[0]) to (X4[0s]);
		\draw[arr] (X2[s]) to (X4[st]);
		\draw[arr] (X2[t]) to (X4[tT]);
		\draw[arr] (X4[0s]) to (X4[s]);
		\draw[arr] (X4[s]) to (X4[st]);
		\draw[arr] (X4[st]) to (X4[t]);
		\draw[arr] (X4[t]) to (X4[tT]);
	\end{tikzpicture}
	\hspace{8pt}
	\begin{tikzpicture}[scale=1,xscale=1.3,yscale=1.2]
		\node[] () at (-2, -.5) {$\overset{\textrm{Def.\ \ref{def:subsampled_dscm}}}{\implies}$};
		\node[] () at (-.4, 1.8) {$G(\Mcal_{\ev(0,s,t)})$:};
		\node[] (W1) at (0, 1) {$W$};
		\node[] (N2) at (1, 1) {$N$};
		\node[] (X2[s]) at (0, -1) {$X_2^{s}$};
		\node[] (X2[t]) at (1, -1) {$X_2^{t}$};
		\draw[arr,color=red!20,draw] (N2) to (X2[s]);
		\draw[arr,color=red!20,draw] (N2) to (X2[t]);
		\node[] (X1[s]) at (0, 0) {$X_1^{s}$};
		\node[] (X1[t]) at (1, 0) {$X_1^{t}$};
		\node[] (X4[s]) at (0, -2) {$X_4^{s}$};
		\node[] (X4[t]) at (1, -2) {$X_4^{t}$};
		\node[] (X1[0]) at (-1, 0) {$X^0_1$};
		\node[] (X2[0]) at (-1, -1) {$X^0_2$};
		\node[] (X4[0]) at (-1, -2) {$X^0_4$};
		\draw[arr] (X1[0]) to (X1[s]);
		\draw[arr] (X2[0]) to (X2[s]);
		\draw[arr] (X1[0]) to (X2[s]);
		\draw[arr] (X2[0]) to (X1[s]);
		\draw[arr] (X4[0]) to (X4[s]);
		\draw[arr,color=red!20,draw] (W1) to (X1[s]);
		\draw[arr,color=red!20,draw] (W1) to (X1[t]);
		\draw[arr] (X1[s]) to (X2[t]);
		\draw[arr] (X2[s]) to (X1[t]);
		\draw[arr] (X2[0]) to (X4[s]);
		\draw[arr] (X2[s]) to (X4[t]);
		\draw[arr,color=red] (X2[s]) to (X4[s]);
		\draw[arr,color=red] (X2[t]) to (X4[t]);
		\draw[arr] (X1[s]) to (X1[t]);
		\draw[arr] (X2[s]) to (X2[t]);
		\draw[arr] (X4[s]) to (X4[t]);
	\end{tikzpicture}
	\caption{The graph $G(\Mcal)$, time-split $G(\Mcal_{\ev^+(0,s,t)})$ and subsampled $G(\Mcal_{\ev(0,s,t)})$. For graphical appeal, the processes $X_1, X_2$ and $X_4$ are assumed to be temporally Markov.}
	\label{fig:dscm_time_evaluations}
\end{figure*}

Given a system of SDEs $\Dcal$, we can construct a system of SDEs $\Dcal_{\ev(r)}$ by splitting up the time interval $[0,T] = [0,r) \cup \{r\} \cup (r,T]$, and consider the SDEs
	\begin{align*}
			X_v(t) & = X_v^0 + \int_0^t g_{v}(s-, X_{\alpha(v)})\diff h_{v}(s, X_{\beta(v)})  \\
			X_v(r) & = X_v^0 + \int_0^r g_{v}(s-, X_{\alpha(v)})\diff h_{v}(s, X_{\beta(v)})  \\
			X_v(u) & = X_v(r) + \int_r^u g_{v}(s-, X_{\alpha(v)})\diff h_{v}(s, X_{\beta(v)})
		\end{align*}
	for all $v\in V$ and all $t\in [0,r)$ and $u\in (r,T]$.
We can do this type of time-splitting also on the level of a DSCM, to be formally defined below.\notefn{For DSCMs induced by SDEs, Definition \ref{def:time-split_dscm} produces none of the edges of the form $X_v^{t}\to X_v^{(t,T]}$ that Figure \ref{fig:dscm_time_evaluations} displays, because the structural equations of $\Mcal_\Dcal$ are It\^o maps in which $v$ is not a parent of itself. \citet{boeken2026causal} provide an alternative definition of the induced DSCM (with self-dependencies) that resolves this issue. The self-loops $v\to v$ that this retains in the summary graph do not affect $\sigma$- or $d$-separations.} The graphs of an example DSCM $\Mcal$ and its time-split version $\Mcal_{\ev^+(0, s, t)}$ and subsampled version $\Mcal_{\ev(0, s, t)}$ are given in Figure \ref{fig:dscm_time_evaluations}.

\begin{definition}[Time-split DSCM]\label{def:time-split_dscm}
	Let a set of distinct time-indices $\tau = \{t_1, ..., t_m\}\subseteq \Tcal$ and DSCM $\Mcal=\angs{V_p \times \Ecal, W, \Xcal_V, \Xcal_W, f, \PP(X_W)}$ be given. The \emph{time-split DSCM} $\Mcal_{\ev^+(\tau)}$ is defined as $\Mcal_{\ev^+(\tau)} = \angs{\tilde{V}, W, \Xcal_{\tilde{V}}, \Xcal_W, \tilde{f}, \PP(X_W)}$ with differences from $\Mcal$:
	\begin{enumerate}[label=\roman*)]
		\item $\tilde{V} = V_p \times \tilde{\Ecal}$, where $\tilde{\Ecal}$ is a further subdivision of $\Ecal$, where for every $t\in \tau$, any interval $[s,u] \in \Ecal$ containing $t$ is partitioned into $\{[s,t), \{t\}, (t,u]\}\subseteq \tilde{\Ecal}$;\footnote{A similar construction applies for (half-)open intervals.}
		\item standard Borel space $\Xcal_{\tilde{V}} = \bigtimes_{(v, \Ical)\in \tilde{V}} D(\Ical, \RR)$;
		\item for every $(v, \Ical) \in \tilde{V}$ a structural equation $\tilde{f}_{v^\Ical}: \Xcal_{\tilde{V}} \times \Xcal_W \to \Xcal_{v^\Ical}$ with for all $t\in\Ical$
		      \begin{align*}
			      \tilde{f}_{v^{\Ical}}(x_{\tilde{V}}, x_W)(t) & := f_v\left(x_{\alpha(v)}^{\tilde{\Ecal}_{< \Ical}}, x_{\beta(v)}^{\tilde{\Ecal}_{\leq \Ical}},  x_W \right)(t),
		      \end{align*}
		      where $\tilde{\Ecal}_{\leq \Ical} := \{\Ical' \in \tilde{\Ecal} \mid \exists s \in \Ical', t\in \Ical :s \leq t\}$.
	\end{enumerate}
\end{definition}
The above definition can straightforwardly be extended to differing time-evaluations $\tau_v$ for each $v\in V$.
Note the explicit modelling of the predictable dependence on $\alpha$, and adapted dependence on $\beta$.
% \red{Note that one can similarly take any DSCM with $\Ecal$ a partition of $\Tcal$, and map it to a DSCM with another partition $\Ecal'$ of $\Tcal$. Split a single interval, merge two intervals, and make further generalisations in terms of thsi }

\begin{definition}[Subsampled DSCM]\label{def:subsampled_dscm}
	Let $\Mcal_{\ev^+(\tau)}$ be a time-split DSCM with endogenous variables $V_p\times \Ecal$. Let $\Ecal' \subseteq \Ecal$ denote the set of non-singleton intervals. The \emph{subsampled DSCM} is defined as $\Mcal_{\ev(\tau)} := (\Mcal_{\ev^+(\tau)})_{\marg(V_p\times \Ecal')}$.
\end{definition}

\begin{example}
	Let $\Mcal_\Dcal$ be the DSCM from Example \ref{ex:dscm}, and let $\Mcal := (\Mcal_\Dcal)_{\marg(L)}$ be the marginalisation of $\Mcal_\Dcal$ with respect to $L :=\{X_3^0, X_3, X_4^0\}$. The graphs $G(\Mcal)$, time-split $G(\Mcal_{\ev^+(s,t)})$ and subsampled $G(\Mcal_{\ev(s,t)})$ are depicted in Figure \ref{fig:dscm_time_evaluations}. We assume the processes $X_1, X_2$ and $X_4$ to be temporally Markov with time-invariant dynamics for graphical appeal.
\end{example}

\begin{definition}[Path concatenation]
	Let $\Ical_X, \Ical_Y$ be two adjacent intervals in $[0,T]$ with $s := \sup \Ical_X = \inf \Ical_Y$, and let $X \in D(\Ical_X, \RR)$ and $Y \in D(\Ical_Y, \RR)$. The concatenation of $X$ and $Y$ is the path $X*Y \in D(\Ical_X \cup \Ical_Y, \RR)$ defined by
	\begin{equation*}
		(X*Y)(t) := \begin{cases}
			X(t)  & \text{if $t \in [0,s) \cap \Ical_X$}  \\
			Y(t+) & \text{if $t\in [s, T] \cap \Ical_Y$.}
		\end{cases}
	\end{equation*}
\end{definition}

\begin{definition}[Collapsed DSCM]\label{def:collapsed_dscm}
	For a DSCM $\Mcal = \angs{V\times \Ecal, W, \Xcal_V, \Xcal_W, f, \PP(X_W)}$, the \emph{collapsed DSCM} is defined as $\Mcal_{*} := \angs{\tilde{V}, W, \Xcal_{\tilde{V}}, \Xcal_W, \tilde{f}, \PP(X_W)}$, where $\tilde{V} := V\times \tilde{\Ecal}$ with $\tilde{\Ecal} := \bigcup \Ecal$ and structural equations
	\begin{equation*}
		\tilde{f}_v(x_V, x_W) := \bigast_{\Ical \in \Ecal} f_{v^\Ical}(x_V, x_W).
	\end{equation*}
\end{definition}

Note that this is consistent with the nomenclature that is introduced earlier, where we referred to a DSCM as collapsed if $\Ecal= \Tcal$.

The following proposition shows that no information is lost when time-splitting a DSCM, but information \emph{is} lost when subsampling a DSCM.
\begin{proposition}\label{thm:time-splitting_consistent}
	Given a DSCM $\Mcal$ with evaluation index $\Ecal$ that partitions $\Tcal$ and given a set of distinct time-indices $\tau = \{t_1, ..., t_m\}\subseteq \Tcal$, we have $\Mcal = \Mcal_{\ev^+(\tau)*}$. On the other hand the map $\Mcal \mapsto \Mcal_{\ev(\tau)}$ is not invertible, and $\Mcal \neq \Mcal_{\ev(\tau)*}$ if $\tau \neq \Tcal$.
\end{proposition}

Analysis of the information-loss in the mapping $\Mcal \mapsto \Mcal_{\ev(\tau)}$ could be particularly interesting in the light of impossibility results for inference of features of $\Mcal$ from the subsampled DSCM $\Mcal_{\ev(\tau)}$.

\subsection{Local Conditional Independence}\label{sec:lci}
Throughout this section, let $\Mcal$ be a simple DSCM with time-interval $\Tcal = [0,T]$ for some $T\in\RR$, with endogenous variables $A, B, C \in V$ such that $X_A, X_B, X_C \in \SSb_{\FF_{\Xcal_W}}$ (i.e.\ they are semimartingales). With respect to $\Fcal^{A, B, C}$, such processes have a unique Doob-Meyer decomposition $X_B = \Lambda_B + M_B$, where `the drift term' $\Lambda_B$ is a predictable $\Fcal^{A, B, C}_t$-adapted process, and `the noise term' $M_B$ is a $\Fcal^{A, B, C}_t$-adapted local martingale.\notefn{This requires $X_B$ to be a \emph{special} semimartingale. For general semimartingales the decomposition is not unique, and $\Lambda_B$ need not be predictable. See \cite{boeken2026causal} for a detailed discussion.} We assume that the process $\Lambda_B$ is absolutely continuous, i.e.\ $\Lambda_B = \int_0^t \lambda_B(s)\diff s$ for some $\Fcal^{A, B, C}_t$-predictable \emph{intensity} process $\lambda_B$.

\begin{definition}[Local conditional independence]\label{def:local_independence}
	We say that $X_B$ is \emph{locally conditionally independent} from $X_A$ given $X_C$, written $X_A \not\to X_B \mid X_C$, if the intensity process $\lambda_B$ has a $\Fcal_{t}^{B, C}$-predictable version.\footnote{If one does not want to assume the existence of the intensity $\lambda_B$, local conditional independence can be defined more generally as $\Lambda_B$ having a $\Fcal^{B, C}_t$-adapted version, or equivalently, that $M_{B}$ has a version that is a $\Fcal^{B, C}_t$-adapted martingale \citep{mogensen2020markov}.}
\end{definition}

In the following theorem, we relate local independence to statistical relations between the variables in time-split DSCMs.

\begin{theorem}\label{thm:dscm_lci}
	Considering the DSCM $\Mcal$ and for every $t\in \Tcal$ the time-split DSCM $\Mcal_{\ev^+(t)}$, we have that\notefn{The first implication of Theorem \ref{thm:dscm_lci} does not hold. If $X_B$ has continuous sample paths, then $X_B^t = X_B^{t-}$ a.s.\ so the conditional independence holds trivially \citep{christgau2023nonparametric}. \citet{boeken2026causal} show instead that $X_A^{[0,s]}\Indep X_B^{(s,t]}\given X_C^{[0,s]}$ for all $s<t$ implies $X_A \not\to X_B \mid X_C$, generalising \cite{florens1996noncausality}.}
	\begin{align*}
		         & X_{B}^t \Indep X_{A}^{[0,t)} | X_{B}^{[0,t)}, X_{C}^{[0,t)} \text{ for all $t\in\Tcal$}                        \\
		\implies & X_A \not\to X_B \mid X_C                                                                                       \\
		\implies & \EE[X_{B}^t \mid X_{A}^{[0,t)}, X_{B}^{[0,t)}, X_{C}^{[0,t)}] = \EE[X_{B}^t \mid X_{B}^{[0,t)}, X_{C}^{[0,t)}] \\
		         & \hspace{165pt}\text{ for all $t\in\Tcal$.}
	\end{align*}
\end{theorem}

\begin{definition}[Local independence graph, \cite{mogensen2022graphical}]\label{def:lci_graph}
	Given a set $V$ of semimartingales $X_v$ with joint distribution $\PP(X_V)$, a directed mixed graph $G=(V, E)$ is a local independence graph for $\PP(X_V)$ if $A \to B$ not in $G(\Mcal)$ implies $X_A \not\to_{\PP} X_B | X_{V\setminus A}$ for all $A, B \in V$.
\end{definition}

Existing work on local independence often assume some kind of \emph{autonomy} of the processes $X_A, X_B, X_C$. \cite{aalen2012causality} assumes that the martingales $M_A, M_B, M_C$ are strongly orthogonal (i.e.\ so have zero quadratic covariation). \cite{didelez2008graphical,roysland2024graphical} consider pure jump-type processes and assume orthogonality, ensuring they never jump simultaneously. \cite{christgau2023nonparametric} require $X_A, X_C$ to be predictable, which ensures $M_A = M_C = 0$. We consider a slightly stronger notion of autonomy:

\begin{assumption}[Independent integrators]\label{ass:independent_integrators}
	Let $\Mcal$ be a collapsed DSCM with $\beta(v) \subseteq W$ and $\beta(v) \cap \beta(v') = \emptyset$ for all $v, v' \in V$.
\end{assumption}

In words, this means that there can be no instantaneous effects between the endogenous variables, nor any instantaneous dependencies between them through latent confounders.

\begin{theorem}\label{thm:dscm_lci_markov_property}
	For a collapsed DSCM $\Mcal$ that satisfies Assumption \ref{ass:independent_integrators} we have:\notefn{The proof is invalid, as it relies on the incorrect Theorem~\ref{thm:dscm_lci}. \cite{boeken2026causal} show that Claim 1 holds under the condition that the exogenous processes enter each mechanism only through the initial condition and through the integrator, which is required to be a martingale. Claim 2 is not established; \citet{boeken2026causal} instead prove a local independence Markov property via $\sigma$-separation in the \emph{time-split} graph.}
	\begin{enumerate}
		\item $G(\Mcal)$ is a local independence graph,
		\item a $\sigma$-separation local independence global Markov property:
		      \begin{align*}
			      A \Perp_{G(\Mcal)}^{\sigma} B | C \implies X_A \not\to X_B \mid X_C.
		      \end{align*}
	\end{enumerate}
\end{theorem}

\begin{example}
	The graph $G(\Mcal_\Dcal)$ in Figure \ref{fig:dscm_graph} is not guaranteed to be a local independence graph, but $G((\Mcal_\Dcal)_{\marg(L)})$ with $L := \{X_3, X_4\}$ is.
\end{example}

We note that the above Markov property is weaker than existing Markov properties for $\delta$ or $\mu$ separation (e.g.\ \cite{didelez2008graphical} for orthogonal counting processes, and \cite{mogensen2018causal} for It\^o processes with independent noise), as $\delta$- and $\mu$-separation imply $\sigma$-separation. However, it applies to any simple DSCM induced by uniquely solvable systems of SDEs with independent integrators, so spans a much larger class of processes.\footnote{\nopagebreak A notable exception is given in \cite{mogensen2022graphical} for Ornstein-Uhlenbeck processes with correlated noise.}

\section{Further Applications}\label{sec:further_applications}
In this section, we briefly elaborate how existing results for causal effect identification and causal discovery can directly be applied to DSCMs.

\subsection{Causal Effect Identification}

We note that perfect interventions can be modelled for an entire trajectory $X_v^{[0,T]}$ in a collapsed DSCM, on a specific subinterval $X_v^{[s, t]}$ with $[s,t]\subseteq \Tcal$ in a time-split DSCM, or on a time-point $X_v^t$ in either a time-split or subsampled DSCM, e.g. in the DSCMs depicted in Figure \ref{fig:dscm_time_evaluations}.

The rules of do-calculus are valid for simple (D)SCMs \citep{forre2020causal}. These criteria are stated in terms of so-called \emph{intervention nodes}, which for variable $X\in V$ are defined as vertices $I_X$ in $G(\Mcal)$ with $I_X \to X$ as their only connected edge.
\begin{theorem}[Do-calculus, \cite{forre2020causal}]\label{thm:do_calculus}
	Given a (D)SCM $\Mcal$ with $X, Y, Z, W \subseteq V$, we have\notefn{Rules 2 and 3 require absolute continuity conditions. See \citet{boeken2026causal}.}
	
	\begin{enumerate}
		\item Insertion/deletion of observation:\\
		$Y\Perp^\sigma X \mid Z, \Do(W) \implies$\\
		\hspace*{15pt}$\PP(Y\mid X, Z, \Do(W)) = \PP(Y\mid Z, \Do(W))$
		\item Action/observation exchange:\\
		$Y\Perp^\sigma I_X \mid X, Z, \Do(W) \implies$\\
		\hspace*{15pt}$\PP(Y\mid \Do(X), Z, \Do(W)) = \PP(Y\mid X, Z, \Do(W))$
		\item Insertion/deletion of actions:\\
		$Y\Perp^\sigma I_X \mid Z, \Do(W) \implies$\\
		\hspace*{15pt}$\PP(Y\mid \Do(X), Z, \Do(W)) = \PP(Y\mid Z, \Do(W))$.
	\end{enumerate}
	
\end{theorem}

As a consequence, generalised adjustment formulae (like backdoor adjustment) based on $\sigma$-separation criteria are also valid for simple DSCMs, as is the ID algorithm. For more information, we refer to \cite{forre2020causal}.

The above mentioned results consider the identification of certain estimands in terms of observational distributions. We note that the construction of \emph{estimators} for such expressions can be highly nontrivial when the variables take values in function spaces.

\subsection{Causal Discovery}

For a given simple DSCM with graph $G$, let $\textrm{IM}_\sigma(G)$ denote the $\sigma$-independence model of $G$, i.e.\ the set of all $\sigma$-separations that are present in $G$. Constraint-based causal discovery algorithms are methods that reconstruct (an equivalence class of) $G$ from the independence model $\textrm{IM}_\sigma(G)$. An example is \emph{Fast Causal Discovery} (FCI) \citep{spirtes1995causal,zhang2008completeness}, which outputs a \emph{Partial Ancestral Graph} (PAG), and which is shown to be sound and complete for simple (and thus possibly cyclic) SCMs. In the following let $F$ denote FCI as a mapping from an independence model to a PAG.

\begin{theorem}[FCI, \cite{mooij2020constraintbased}]
	Let $\Mcal$ be a simple (D)SCM with graph $G$, then
	
	\begin{enumerate}
		\item FCI is sound: $G \in F(\textrm{IM}_\sigma(G))$
		\item FCI is complete: let $G'$ be another DMG, then $F(\textrm{IM}_\sigma(G)) = F(\textrm{IM}_\sigma(G'))$ iff $\textrm{IM}_\sigma(G) = \textrm{IM}_\sigma(G')$.
	\end{enumerate}
	
\end{theorem}

Other constraint-based causal discovery algorithms that are known to be sound for (possibly cyclic) simple DSCMs are Local Causal Discovery \citep{cooper1997simple,mooij2020joint} and Y-structures \citep{mani2006bayesian,mooij2020joint}.

In practice, one requires a conditional independence (CI) test to map a dataset with samples of an SCM $\Mcal$ to an independence model (assuming faithfulness). If variables $X, Y, Z$ take values in $D(\Tcal, \RR)$, testing $X\Indep Y\given Z$ is not straightforward. Recently, first results for such functional CI testing are proposed in \cite{lundborg2022conditional,laumann2023kernelbased} and \cite{manten2024signature}.

\begin{example}
	Let $\Mcal_\Dcal$ be the DSCM from Example \ref{ex:dscm}. The output of FCI on the $\sigma$-separation independence model of $G((\Mcal_\Dcal)_{\ev^+(0)})$ is given in Figure \ref{fig:fci_output}.
\end{example}

\begin{figure}[!hbt]
	\centering
	\begin{tikzpicture}[xscale=1.4,yscale=1.2]
		\node[var] (X1) at (0, 0) {$X_1$};
		\node[var] (X2) at (-0.5, -1) {$X_2$};
		\node[var] (X3) at (.5, -1) {$X_3$};
		\node[var] (X4) at (0, -2) {$X_4$};
		\node[var] (X1[0]) at (-1.5, 0) {$X^0_1$};
		\node[var] (X2[0]) at (-1.5, -1) {$X^0_2$};
		\node[var] (X3[0]) at (1.5, -1) {$X^0_3$};
		\node[var] (X4[0]) at (-1.5, -2) {$X^0_4$};
		\draw[carc] (X1) to (X2);
		\draw[carc] (X2) to (X3);
		\draw[carc] (X3) to (X1);
		\draw[arr] (X2) to (X4);
		\draw[car] (X1[0]) to (X1);
		\draw[car] (X1[0]) to (X2);
		\draw[car] (X1[0]) to (X3);
		\draw[car] (X2[0]) to (X1);
		\draw[car] (X2[0]) to (X2);
		\draw[car,bend right=40] (X2[0]) to (X3);
		\draw[car] (X3[0]) to (X1);
		\draw[car,bend left=40] (X3[0]) to (X2);
		\draw[car] (X3[0]) to (X3);
		\draw[car] (X4[0]) to (X4);
	\end{tikzpicture}
	\caption{The output of FCI on $\textrm{IM}_\sigma(G((\Mcal_\Dcal)_{\ev^+(0)}))$.}
	\label{fig:fci_output}
\end{figure}

We conjecture that `Tiered FCI' \citep{andrews2020completeness} is sound for cyclic DSCMs, and applying it on evaluated DSCM $\Mcal_{\ev(t)}$ and subsequently mapping its output back to the PAG for $\Mcal$ can be used to orient more edges in the output of regular FCI.

\section{Conclusion}
In this work, we have refined the concept of a \emph{Dynamic Structural Causal Model}. We have shown that a large class of systems of differential equations can be mapped to a DSCM, giving it a Markov property based on $\sigma$-separation. We have formalised the concepts of \emph{time-splitting} and \emph{subsampling}, which both subsequently result in a DSCM, and thus remain to have clear causal semantics. We have shown how a time-split DSCM can shed some light on the concept of local conditional independence, and that the graph of a collapsed DSCM can be interpreted as a local independence graph if all integrators are exogenous and independent. Additionally, we have shown how existing results for causal effect identification and constraint-based causal discovery for simple SCMs can directly be applied to DSCMs.

% \paragraph{Acknowledgements}
% The authors thank Søren Mogensen and Alexander Christgau for pointers on Local Independence.
% Simon Ferreira and Charles Assaad for d-separation Markov property.

\bibliographystyle{apalike}
\setlength{\bibsep}{2pt}
{\small\bibliography{refs}}

\end{document}